\documentclass[10pt]{article}
\usepackage{amstext, amsmath, amsthm, amssymb}
\usepackage{amsmath}
\usepackage{geometry} 
\usepackage{color}

\numberwithin{equation}{section}
\usepackage[nottoc]{tocbibind}
\usepackage{hyperref}
\usepackage{makeidx}
\makeindex

\newtheorem{thm}{Theorem}[section]
\newtheorem{cor}[thm]{Corollary}
\newtheorem{lem}[thm]{Lemma}

\newtheorem{defin}[thm]{Definition}
\newtheorem{rem}[thm]{Remark}
\newtheorem{example}[thm]{Example}

\usepackage{colortbl}

\def\suml{\sum\limits}
\def\intl{\int\limits}

\def\id{\mathop{Id}}

\title{A Jordan-Schwinger Variant of the Spectral Theorem
for Linear Operators}

\author{Wolfgang Bock, Vyacheslav Futorny, Mikhail Neklyudov}

\AtEndDocument{\bigskip{\footnotesize%
  (W.\,Bock) \textsc{Technomathematics Group,
University of Kaiserslautern,
P. O. Box 3049, 67653 Kaiserslautern, Germany} \par
  \textit{E-mail address}: \texttt{bock@mathematik.uni-kl.de} \par
  \addvspace{\medskipamount}
 (V.\,Futorny) \textsc{SUSTech, Shenzhen, China} and \textsc{University of S\~ao Paulo, Brasil} \par
  \textit{E-mail address}: \texttt{vfutorny@gmail.com}  \par
   \addvspace{\medskipamount}
  (M.\,Neklyudov) \textsc{Instituto de Ci\^{e}ncias Exatas, Departamento de Matematica, UFAM, Manaus, CEP 69077-000, Brasil} \par
  \textit{E-mail address}: \texttt{misha.neklyudov@gmail.com}
}}

\date{\today}

\begin{document}
\maketitle
\begin{abstract}
In this paper we show variant of the spectral theorem using an algebraic Jordan-Schwinger map. The advantage of this approach is that we don't have restriction of normality on the class of operators we consider. On the other side, we have the restriction that the class of operators we consider should be of weighted Hilbert-Schmidt class. 
\end{abstract}

\section{Introduction}

The spectral theorem is one of the most fundamental theorems in linear functional analysis. It has a tremendous impact due to the various applications that the powerful toolbox of spectral decomposition is providing. Furthermore the description of physical systems and especially their organization and long-range behaviour is encoded in the spectral values \cite{Tre11}. Especially in quantum physics and thermodynamics, where the energy values are eigenvalues of a suitable Hamiltonian this particular importance becomes clear.
Spectral theory, however has its applications far beyond physics. To a large extend it is used to describe the evolution according to a linear dynamics in a suitable space. It hence has direct interlinks into the classical semi-group theory \cite{Pazy, CurtainZwart, JacobZwart} and is therefore massively used in systems theory. A natural involvement via semi-group methods can then also be found in stochastic analysis, in particular in the theory of Markov processes and Markov chains as well as Dirichlet forms.
In recent years spectral theory has also been used in machine learning and artificial intelligence, in particular in pattern recognition. A spectral decomposition allows a principal component analysis \cite{PCA} and hence simplifies the number of contributing modes in a pattern by just considering the dominant ones. In face recognition one for example then talks about eigenfaces \cite{eigenfaces}.
In line of the early developments of Hilbert's operator theory, see e.g.~\cite{CHilbert} in the early 30s of the last century von Neumann made important contributions to the spectral theorem for normal operators \cite{vN1,vN2}. Moreover it was used as the mathematical foundation of many progressing theories in that time from partial differential equations to quantum physics.  
The spectral theorem itself is understood as a collection of results on normal operators; see, e.g. \cite{Hass}. These are generalizing the finite dimensional case of a diagonalizable or triagonalizable matrix. 
The generalization to non-normal operators even beyond the compact case is much more involved and less direct than in the standard case. First attempts were undertaken by Gonshor et al. \cite{Go56, Go58} for special classes of non-linear operators. For special perturbations of normal operators a spectral theorem was given in \cite{Ding2008}. 

In this paper we show variant of the spectral theorem using an algebraic Jordan-Schwinger map introduced in \cite{BFN2021} (Definition 7.2). The advantage of this approach is that we don't have restriction of normality on the class of operators we consider. On the other side, we have the restriction that the class of operators we consider should be of weighted Hilbert-Schmidt class.

\section{Preliminaries}
Fix a separable infinite dimensional Hilbert space $H$.
Let $S_0, S_1:H\to H$ be generators of the Cuntz algebra $O_2$. We will assume that $S_0$ is a shift operator in the sense of \cite{BratelliJoergensen97}, that is $$\cap_{n=1}^{\infty} S_0^n(H)=\{0\}.$$ Note that then $K:=H-S_0(H)$ is a separable Hilbert space. Hence we can choose an orthonormal basis $\{\xi_{0,j}\}_{j=1}^{\infty}$ of $K$ and define $$\xi_{i,j}:=S_0^{i}\xi_{0,j},\quad i\in \mathbb{N}\cup\{0\},\,j\in\mathbb{N}.$$ Then $\{\xi_{i,j}\}$ is an orthonormal basis of $H$. Denote by $\mathbb{H}^2(\mathbb{T})$ the Hardy subspace of square integrable functions on the unit torus $\mathbb{T}$, denoted by $L^2(\mathbb{T})$ and define a unitary operator
\[
V: H\to \mathcal{H}_{+}(K):=\mathbb{H}^2(\mathbb{T})\otimes K,\quad V \xi_{i,j}:=z^{i+1}\otimes \xi_{0,j},\quad i\in \mathbb{N}\cup\{0\},j\in\mathbb{N}.
\]
We identify the elements of $\mathcal{H}_{+}(K)$ with functions from $\mathbb{T}$ to $K$ in such a way that 
$$f\in \mathcal{H}_{+}(K) \text{ if and only if } f(z)=\suml_{n=1}^{\infty}f_n z^n, \quad f_n\in K.$$ Then $\mathcal{H}_{+}(K)$ is a Hilbert space with scalar product
\[
(f,g)_{\mathcal{H}_{+}(K)}= \suml_{n=1}^{\infty}(f_n,g_n)_{K}=\frac{1}{2\pi i}\intl_{\mathbb{S}^1}(f(z), \bar{g}(z))_{K}\frac{dz}{z}.
\]
We will also make use of a Sobolev space based on $\mathcal{H}_{+}(K)$. To define it let us introduce the operator
\[
J(f)(z)=\frac{1}{z}\intl_0^z f(w)\, dw, f\in \mathcal{H}_{+}(K).
\]
Notice that
\[
||J^p(f)||_{\mathcal{H}_{+}(K)}^2=\sum\limits_{n=1}^{\infty}\frac{||f_n||_K^2}{(n+1)^{2p}},p\in\mathbb{Z}.
\]
Moreover, $J$ is invertible with $(z\frac{d}{dz}+1)J=\id$. Thus we can define a Hilbert space
\[
\mathcal{H}_{+}^p(K):=\{f: J^p(f)\in\mathcal{H}_{+}(K),\mbox{ such that }||J^p(f)||_{\mathcal{H}_{+}(K)}<\infty\},\quad p\in\mathbb{Z}.
\]

Notice that $S_0^{+}=V S_0 V^*$ is a multiplication operator, i.e.: 
\begin{equation}\label{eqn:S_0}
S_0^{+}\xi(z)=z\xi(z),\quad \xi\in \mathcal{H}_{+}(K),
\end{equation}
and, correspondingly, $S_0^*$ is unitary equivalent to  
\begin{equation}\label{eqn:S_0adj}
(S_0^{+})^{*}\xi(z)=V S_0^* V^*=\frac{\xi(z)}{z}-\frac{1}{2\pi i}\intl_{\mathbb{S}^1}\frac{\xi(w)}{w^2}dw=\bar{z}\xi(z)-\frac{1}{2\pi i}\intl_{\mathbb{S}^1}\bar{w}^2\xi(w)dw,\xi\in \mathcal{H}_{+}(K).
\end{equation}
In order to establish a connection with Cuntz algebras,
we recall the following definition of the Jordan-Schwinger map (see \cite{BFN2021}, Definition 7.2):
\begin{defin}\label{def:JSmap}
Let $V$ be a separable locally convex Hausdorff topological vector space  in duality $<\cdot,\cdot>$ with $V^*$ and biorthogonal system $\{e_k,f_k\}_{k=1}^{\infty}$. Define the Jordan-Schwinger map as follows  
\begin{equation}\label{ex_1}
D(A):=\sum\limits_{\alpha,\beta\in \mathbb{N}}<Ae_{\alpha},f_{\beta}> s_{\alpha} s_{\beta}^*, \quad A\in \mathcal{L}(V,V),
\end{equation}
\begin{equation}\label{ex_1a}
\partial(f):= \sum\limits_{\beta\in \mathbb{N}}<f,f_{\beta}> s_{\beta}^*, f\in V,
\end{equation}
\begin{equation}\label{ex_1b}
\bar{\partial}(g):= \sum\limits_{\alpha\in \mathbb{N}}<e_{\alpha},g> s_{\alpha}, g\in V^*
\end{equation}
where $\{s_k, s_k^*\}_{k=1}^{\infty}$ are generators of the Cuntz algebra  $O_{\infty}$.
\end{defin}

We also recall the following result \cite{BratelliJoergensen97} (Lemma 6.1), which gives a one-to-one correspondence between representations of th Cuntz algebras $O_{\infty}$ and $O_N$.

\begin{lem}\label{lem:O_fin_infin}
There is 1-to-1 correspondence between representations $s_i\mapsto S_i,i=0,\ldots ,N-1$ of $O_N$ on $\mathcal{H}$ such that $S_0$ is a shift, and representations of $O_{\infty}$ on $\mathcal{H}$ such that the sum of the ranges of
the isometries is $\id$. If the representatives of the generators of $O_{\infty}$ are denoted by $T_k^{j}$, $j=1,\ldots,N-1$, $k\in\mathbb{N}$, so that
\[
(T_{k_1}^{j_1})^* T_{k_2}^{j_2}=\delta_{k_1,k_2}\delta_{j_1,j_2}\id,
\]
and
\[
\sum\limits_{j=1}^{N-1}\sum\limits_{k=1}^{\infty}T_k^{j}(T_k^{j})^*=\id,
\]
then the 1-to-1 correspondence is given by
\[
T_k^j=S_0^{k-1}S_j, j=1,\ldots,N-1, k\in\mathbb{N},
\]
and by
\[
S_0=\sum\limits_{j=1}^{N-1}\sum\limits_{k=1}^{\infty}T_{k+1}^{j}T_k^{j,*},
\]
and
\[
S_j=T_1^j, j=1,\ldots,N-1,
\]
where all infinite sums converge in strong operator topology.
\end{lem}
We will apply Lemma \ref{lem:O_fin_infin} for $N=2$. Since in this case we have $j=1$, we omit the index $j$ for the sake of readability.

\section{Spectral Theorem}

With the help of Lemma \eqref{lem:O_fin_infin} for $N=2$ together with the definition \ref{def:JSmap} of the map $D$ we can connect every linear bounded operator to an analytic function in the variables $S_0$ and $S_0^*$. This can already be considered as a version of the spectral theorem. We will show throughout the section that indeed $D(A)$ and $A$ have the same spectrum and furthermore we have a representation of the quadratic form of the operator $A$ in terms of $D(A)$. 
\begin{lem}\label{lem:Spectral}
\[
D(A)=\sum\limits_{m,n\in\mathbb{N}}<A e_m, f_n>S_0^{m-1}S_1 S_1^* (S_0^*)^{n-1}, 
\]
or, equivalently (applying the identity $S_0 S_0^*+S_1 S_1^*=\id$),
\begin{eqnarray}
D(A)&=\sum\limits_{m,n\in\mathbb{N}}(<A e_{m+1}, f_{n+1}>-<A e_m, f_n>)S_0^{m} (S_0^*)^{n}\nonumber\\
&+\sum\limits_{m=0}^{\infty}<A e_{m+1}, f_{1}>S_0^{m}+\sum\limits_{n=1}^{\infty}<A e_{1}, f_{n+1}>(S_0^*)^{n}.
\end{eqnarray}
\end{lem}

Now we can apply formulas \eqref{eqn:S_0}, \eqref{eqn:S_0adj} to deduce the following unitary equivalence of $D(A)$ with an analytic function of a shift operator on the Hardy space and its adjoint. As a result we obtain that $D(A)$ is unitary equivalent to a suitable integral operator.  
\begin{thm}\label{thm:Spectral}
The operator $D(A)$ is unitary equivalent to the operator $F(S_0^{+},(S_0^{+})^{*})$ where \\ $F=F(z,w):\mathbb{C}^2\to\mathbb{C}$ given by
\[
F=\sum\limits_{m,n\in\mathbb{N}}(<A e_{m+1}, f_{n+1}>-<A e_m, f_n>)z^m w^n+\sum\limits_{m=0}^{\infty}<A e_{m+1}, f_{1}> z^m+\sum\limits_{n=1}^{\infty}<A e_{1}, f_{n+1}> w^n,
\]
and the operators $S_0^{+}$, $(S_0^{+})^{*}$ are given by formulas \eqref{eqn:S_0} and \eqref{eqn:S_0adj} respectively.
Moreover, we have
\begin{equation}\label{eqn:ReprIntegral}
VD(A)V^*f(z)=\frac{1}{2\pi i}\intl_{\mathbb{S}^1} K(A,z,w)\frac{f(w)}{w}dw, 
\end{equation}
with
\[
K(A,z,w):=\sum\limits_{m,n=1}^{\infty}<A e_m, f_n> z^m w^{-n}.
\]
Furthermore, if  $A$ is an $(r,p)$-weighted Hilbert-Schmidt operator, i.e.
\begin{equation}\label{eqn:wHS}
||A||_{HS^{(r,p)}}^2=\sum\limits_{m,n=1}^{\infty}\frac{|<A e_m,f_n>|^2(n+1)^{2r}}{(m+1)^{2p}}<\infty,
\end{equation}
then
\begin{equation}\label{eqn:NormEstimate}
||VD(A)V^*||_{\mathcal{L}(\mathcal{H}_{+}^r(K),\mathcal{H}_{+}^p(K))}^2\leq ||A||_{HS^{(r,p)}}^2
\end{equation}
\end{thm}
\begin{proof}
Let us show formula \eqref{eqn:ReprIntegral}. We have
\[
VD(A)V^*= \sum\limits_{m,n\in\mathbb{N}}<A e_m, f_n>(S_0^+)^{m-1}(\id-S_0^+(S_0^+)^*)((S_0^+)^*)^{n-1}.
\]
Now it is easy to calculate that
\[
(\id-S_0^+(S_0^+)^*)g(z)=\frac{z}{2\pi i}\intl_{\mathbb{S}^1}\frac{g(w)}{w^2}\,dw.
\]
Consequently, we have
\[
VD(A)V^*f(z)= \frac{1}{2\pi i}\sum\limits_{m,n\in\mathbb{N}}<A e_m, f_n>z^m \intl_{\mathbb{S}^1}\frac{((S_0^+)^*)^{n-1} f}{w^2}\,dw.
\]
Now the result follows from the representation
\[
((S_0^+)^*)^{n-1} f(z)=\frac{f(z)}{z^{n-1}}+\sum\limits_{i=1}^{n-1}\frac{\kappa_i(f)}{z^{n-i-1}},
\] 
where $\{\kappa_i\}_{i=1}^{n-1}$ are certain linear functionals of $f$, and the observation that only the first term remains after the integration, that is
\[
\intl_{\mathbb{S}^1}\frac{((S_0^+)^*)^{n-1} f}{w^2}\,dw=\intl_{\mathbb{S}^1}\frac{f(w)}{w^{n+1}}\,dw.
\]
The inequality \eqref{eqn:NormEstimate} can be obtained as follows. 
\begin{eqnarray*}
&\left|\frac{1}{2\pi i}\intl_{\mathbb{S}^1} I^p K(A,z,w)\frac{f(w)}{w} dw\right|_{\mathcal{H}_{+}(K)}^2 = \left|\frac{1}{2\pi} \intl_{0}^{2\pi} I^p K(A,z,e^{i\phi})f(e^{i\phi})\,d\phi\right|_{\mathcal{H}_{+}(K)}^2\\
&= \left|\suml_{m,n=1}^{\infty}\frac{A_{mn}}{(m+1)^{p}}z^m
\frac{1}{2\pi} \intl_{0}^{2\pi}f(e^{i\phi}) e^{-i n\phi}\,d\phi
\right|_{\mathcal{H}_{+}(K)}^2\\
&= \left|\suml_{m=1}^{\infty}\left(\suml_{n=1}^{\infty}\frac{A_{mn}}{(m+1)^{p}}f_n\right) z^m
\right|_{\mathcal{H}_{+}(K)}^2\\
&=\suml_{m=1}^{\infty}\left|\suml_{n=1}^{\infty}
\frac{A_{mn}(n+1)^{r}}{(m+1)^{p}}\frac{f_n}{(n+1)^{r}}\right|_{K}^2\\
&\leq \sum\limits_{m,n=1}^{\infty}\frac{|<A e_m,f_n>|^2(n+1)^{2r}}{(m+1)^{2p}} \left| f\right|_{\mathcal{H}_{+}^r(K)}^2.
\end{eqnarray*}
Which proves the assertion.

\end{proof}
\begin{rem}
Suppose that an operator $A$ satisfy the following conditions:
\[
|<A e_m,f_n>|\leq \frac{C}{1+|n-m|^l}, l>r+\frac{1}{2}
\]
and
\[
p>r+\frac{1}{2}.
\]
Then the operator $A$ belongs to the weighted Hilbert-Schmidt space $HS^{(r,p)}$.
\end{rem}
\begin{example}
\begin{trivlist}
\item[(a)] Let $A$ be the following infinite diagonal matrix
\[
A=\left(
\begin{array}{ccccc}
 \alpha    & 0 & \ldots & 0 & \ldots\\
  0  & \frac{\alpha^2}{2} & 0\ &\ldots & \ldots\\
  0 & 0 & \frac{\alpha^3}{3} & 0 & \ldots\\
  \ldots & \ldots & \ldots & \ldots & \ldots\\
  0 & \ldots & 0 &\frac{\alpha^n}{n} & \ldots\\
  \ldots & \ldots & \ldots & \ldots & \ldots
\end{array}
\right),\, |\alpha|<1.
\]
Then we can deduce from Theorem \ref{thm:Spectral} that the corresponding operator $D(A)$ is unitary equivalent to the following operator
\[
f\mapsto \frac{1}{2\pi}\intl_{0}^{2\pi} \ln{(1-\alpha e^{i(\phi-\psi)})} f(e^{i\psi})\,d\psi
\]
\item[(b)] 
If
\[
A=\left(
\begin{array}{ccccc}
 1    & 0 & \ldots & 0 & \ldots\\
  0  & \frac{1}{2} & 0\ &\ldots & \ldots\\
  0 & 0 & \frac{1}{4} & 0 & \ldots\\
  \ldots & \ldots & \ldots & \ldots & \ldots\\
  0 & \ldots & 0 &\frac{1}{2^n} & \ldots\\
  \ldots & \ldots & \ldots & \ldots & \ldots
\end{array}
\right),\, |\alpha|<1,
\]
then similarly one deduces that $D(A)$ is unitary equivalent to the operator
\[
f\mapsto \frac{1}{2\pi}\intl_{0}^{2\pi} \frac{1}{(1-\frac{1}{2} e^{i(\phi-\psi)})} f(e^{i\psi})\,d\psi.
\]
\item[(c)] If $A$ is an operator in the Jordan form with eigenvalues $\{a_m\}_{m=1}^{\infty}$ and corresponding sizes $\{n_k-n_{k-1}\}_{k=1}^{\infty}$ of the Jordan blocks ($n_0=0$ by definition), i.e.
\[
A=J_{n_1}(a_1)\oplus J_{n_2-n_1}(a_2)\ldots J_{n_{m}-n_{m-1}}(a_m)\oplus .
\]
Then $D(A)$ is unitary equivalent to the operator which maps $f$ to

\[\frac{1}{2\pi}\intl_{0}^{2\pi} \left[\frac{e^{i(\phi-\psi)})}{e^{i(\phi-\psi)})-1}\suml_{l=1}^{\infty}a_l\left(e^{i n_l(\phi-\psi)}-e^{i n_{l-1}(\phi-\psi)}\right)+\\ e^{i\phi}\left(\frac{e^{i(\phi-\psi)}}{1-e^{i(\phi-\psi)}}-\kappa(e^{i(\phi-\psi)})\right)\right] f(e^{i\psi})\,d\psi,
\]
where $\kappa(z)=\suml_{k=1}^{\infty} z^{n_k}$. For instance, if all $n_k=k, k\in\mathbb{N}$, that is the operator is diagonal, we get the kernel of the form $\suml_{l=1}^{\infty}a_l e^{i l(\phi-\psi)}$. If $A$ is not  diagonal then the corresponding integral operator is not  a convolution operator. 
\item[(d)] Let $A=\{a_{i-j}\}_{i,j\in\mathbb{Z}}$ be the double sided Toeplitz operator then the  integral operator corresponding to $D(A)$ will have kernel
\[
\frac{e^{i(\phi-\psi)}}{1-e^{i(\phi-\psi)}}\suml_{k\in\mathbb{Z}} a_k e^{ik\phi}.
\]
\item[(e)] Let $A=\{a_{i-j}\}_{i,j=1}^{\infty}$ be the Toeplitz operator. Then the  integral operator corresponding to $D(A)$ will have kernel
\[
\suml_{m,n=1}^{\infty} a_{m-n} e^{i(m\phi-n\psi)}.
\]
\end{trivlist}
\end{example}

The following corollary gives a one-to-one connection between the bilinear form of the operator $A$ and its Jordan-Schwinger image. 

\begin{cor}
We have the following representation for $A$
\[
<Af,g>=\partial(f)D(A)\bar{\partial}(g)=\partial(f)V^* F(S_0^{+},(S_0^{+})^{*})V\bar{\partial}(g).
\]
\end{cor}
For an operator $A$ we denote $\sigma(A)$ (corr. $\sigma_p(A)$, $\sigma_{res}(A)$) its spectrum (corr. point spectrum, residue spectrum).
The following Lemma will show how the spectra of $A$ and its different parts connect to those of $D(A)$. Indeed we see that $D(A)$ is rather related to $A^*$ from spectral point of view.
\begin{lem}\label{lem:Spectrum}
 For $A\in \mathcal{L}(V,V)$ we have 
\[
\sigma_p(A)\subset \sigma_{res}(D(A))\cup\sigma_{p}(D(A)), \sigma_p(A^*)\subset \sigma_{p}(D(A)).
\]
Moreover,
\[
\sigma_{p}(D(A))\subset \sigma_{res}(A)\cup\sigma_{p}(A).
\]
If, in addition, $V$ is a Frechet space then
\[
\sigma(A)=\sigma(D(A)). 
\]
Furthermore, if $V$ is a separable Banach space with separable dual then we have 
\[
D(A^*)=D^*(A).
\]

\end{lem}
\begin{proof}
\begin{itemize}
\item[(a)]
If $\lambda\in \sigma_p(A)$ and, consequently, $Af=\lambda f, f\in V$, then we have $\partial(f)D(A)=\lambda\partial(f)$ by Lemma 7.3 in \cite{BFN2021}. Hence
\[
\partial(f)(D(A)-\lambda \id)=0,f\neq 0
\]
and the image of the operator $(D(A)-\lambda \id)$ is not dense in $H$.
\item[(b)] Similarly to (a), for $\lambda\in \sigma_p(A^*)$ and a corresponding eigenvector $g\in V^*$ we have by Lemma 7.3 in \cite{BFN2021} that
\[
(D(A)-\lambda \id)\bar{\partial}(g)=0.
\]
Consequently, $\lambda\in \sigma_{p}(D(A))$.
\item[(c)]
If $\lambda\in \sigma_p(D(A))$ and $\psi$ is a corresponding eigenvector then for any $f\in V$ we have by Lemma 7.3 in \cite{BFN2021} that
\[
\partial(Af)\psi=\lambda\partial(f)\psi,
\]
i.e. $\partial((A-\lambda \id)f)\psi=0, f\in V$. Thus an image of the operator $(A-\lambda \id)$ is not dense in $V$.
\item[(d)] Follows from Corollary 7.5 in \cite{BFN2021}.
\item[(e)] To show the equality $D(A^*)=D^*(A)$ it is enough to notice that by the classical result of Markushevich \cite{M1943} there exists a shrinking Markushevich basis. 

\end{itemize}

\end{proof}

\begin{rem}
In the theorems, lemmata and corollaries before, in which we discussed the spectrum, we did not assume that the operator is normal, symmetric or has any other properties than being bounded in locally convex topological separable Hausdorff space. In particular $A$ just needs to fulfill that it is weighted Hilbert-Schmidt operator on a space with a biorthogonal system, see the defnition of the norm in  \ref{eqn:wHS} . 
\end{rem}
\begin{cor}
Assume that $V$ is a separable Banach space. If $A\in \mathcal{L}(V,V)\cap HS^{(r,p)}$ is a compact operator, then $D(A)$ is compact. Furthermore, if $A$ belongs to Schatten class with regularity $p\geq 1$, then $D(A)$ belongs to Schatten class with the same regularity $p$.
\end{cor}
\begin{proof}
Let $A\in HS^{(r,p)}$ be a compact operator. Then $D(A)\in \mathcal{L}(\mathcal{H}_{+}^r(K),\mathcal{H}_{+}^p(K)))$ and can be approximated by finite rank projections in the same norm by inequality \eqref{eqn:NormEstimate}. Thus $D(A)$ is compact. By Lemma \ref{lem:Spectrum} $\sigma(D(A))=\sigma(A)=\sigma_p(A)\cup\{0\}$ and, moreover, $\sigma_p(D(A))\subset \sigma_p(A)\cup \sigma_{res}(A)=\sigma_p(A)\cup\{0\}$. Hence the point spectrum of $D(A)$  is discrete with the only possible accumulation point in $\{0\}$ and subset of the point spectrum of $A$ (up to point $\{0\}$). Since Schatten class norms depends on the point spectrum only  and $\sigma_p(D(A))\subset \sigma_p(A)$ we get the second part of the result.
\end{proof}

\subsection{Homeomorphic embedding of a countable group $G$}

Let $G$ be a countable group of order $|G|$ (which could be infinity), $\mathcal{F}(G)$ a space of all real valued functions on $G$, $H$ a separable Hilbert space of infinite dimension, $O(H)$ the group of orthogonal operators on $H$ and $\{S_{g}, g\in G\}$ the generators of Cuntz algebra $O_{|G|}$ of mutually orthogonal isometries of $H$. We enumerate generators $S_{\cdot}$ with elements of $G$. We can assume without loss of generality that
\[
\sum\limits_{h\in G} S_hS_h^*=Id.
\]

\begin{defin}\label{def:Representation}
Define $D:G\to O(H)$, $\partial,\bar{\partial}:\mathcal{F}(G)\to O_{|G|}$,  as follows
\[
D(g):=\sum\limits_{h\in G} S_h S_{gh}^*,\partial(\alpha):=\sum\limits_{h\in G} \alpha(h)S_h, \quad \bar{\partial}(\alpha):=\sum\limits_{h\in G} \alpha(h)S_h^*, \quad g\in G,\alpha\in \mathcal{F}(G).
\]
\end{defin}
The next lemma shows the well-definedness of $D$ on a group.
\begin{lem}\label{lem:PropHom}
The map $D$ is a faithful homomorphism of  groups: 
\begin{equation}\label{eqn:hom}
D(g)D(f)=D(gf), D^*(g)=D(g^{-1}), D(e)=Id,g,f\in G.
\end{equation}
Furthermore,
\begin{equation}\label{eqn:hom_2}
D(g)\partial(\alpha)=\partial(\alpha\circ g),\bar{\partial}(\beta)D(g)=\bar{\partial}(\beta\circ g^{-1}),
\bar{\partial}(\beta)\partial(\alpha)=\sum\limits_{h\in G}\alpha(h)\beta(h),\alpha,\beta\in \mathcal{F}(G),g\in G.
\end{equation}
\end{lem}
\begin{proof}
The result immediately follows from definitions of $D$, $\partial$, $\bar{\partial}$.
\end{proof}

Since the operators $D(g), g\in G$ are orthogonal the spectral theorem for normal operators gives that 
\[
D(g)=\int\limits_{\sigma(D(g))} \lambda d E_{\lambda}(g),
\]
where $E_{\lambda}(g), \lambda\in\sigma(D(g))$ are spectral projections of $D(g)$. Consequently, for any $f\in L^2(\mathbb{S}^1)$ we can define 
\[
f(D(g)):=\int\limits_{\sigma(D(g))} f(\lambda) d E_{\lambda}(g).
\]
In this way, we can expand the group $G$ by adding, for example, all square roots of elements of $G$.
\begin{rem}
The representation in the definition \ref{def:Representation}  can be written using the construction of Lemma \eqref{lem:O_fin_infin} as follows. Let $G$ be an infinite countable group and $K:G\to \mathbb{N}$ 
be an arbitrary enumeration of elements of $G$ and $S_0:H\to H$ be a shift operator on $H$ then, in the same way as in the Lemma \ref{lem:Spectral}, we get the following formula
\begin{equation}\label{eqn:reprconstruction}
D(g):=\sum\limits_{h\in G}S_0^{K(h)-1}(S_0^*)^{K(gh)-1}-S_0^{K(h)}(S_0^*)^{K(gh)}, g\in G
\end{equation}
Every enumeration $K$ can be written as $K=\sigma \circ K_0$ where $K_0$ is some fixed enumeration and $\sigma:\mathbb{N}\to\mathbb{N}$ 
is a transposition.

Thus the representation \eqref{eqn:reprconstruction} is characterised by two parameters-- shift operator $S_0$
of $H$ and an element $\sigma$ of the group of transpositions.

%
%

\end{rem}
\begin{example}
Let $G=\mathbb{Z}_2\times \mathbb{Z}_2$ be the Klein group with generators $a,b$ and $S_e,S_a,S_b,S_{ab}$ be the mutually orthogonal isometries of auxiliary infinitedimensional separable Hilbert space $H$ generating  Cuntz algebra $O_4$. Then we have the following faithfull representation of $G$:
\[
D(e):=Id, D(a):=S_e S_a^*+S_a S_e^*+S_b S_{ab}^* + S_{ab} S_b^*,
\]
\[
D(b):=S_eS_b^*+S_b S_e^*+S_a S_{ab}^*+S_{ab} S_a^*, 
\]
\[
D(ab):=S_e S_{ab}^*+S_{ab} S_e^*+S_a S_b^*+S_b S_a^*.
\]
\end{example}

\begin{example}
Let $G=\mathbb{Z}$ and $K:\mathbb{Z}\to \mathbb{N}$ given by
\[
K(m):=\left\{
\begin{array}{rcl}
2m+1 & , & m\in\mathbb{N}\cup\{0\}\\
-2m & , & m<0
\end{array}
\right.
\]
Then $D(0)=I$ and for $m\in\mathbb{N}$,
\begin{eqnarray}
D(m)&=\suml_{k=1}^m S_0^{2k-1}(I-S_0 S_0^*)(S_0^*)^{2m-2k}\nonumber\\
&+
\suml_{k=m+1}^{\infty} S_0^{2k-1}(I-S_0 S_0^*)(S_0^*)^{2k-2m-1}\nonumber\\
&+\suml_{n=0}^{\infty} S_0^{2n}(I-S_0 S_0^*)(S_0^*)^{2n+2m},
\end{eqnarray}
 $D(-m)=D(m)^*$ or, direct calculation gives
\begin{eqnarray}
D(-m)&=\suml_{n=0}^m S_0^{2n}(I-S_0 S_0^*)(S_0^*)^{2(m-n)-1}\nonumber\\
&+
\suml_{n=m+1}^{\infty} S_0^{2n}(I-S_0 S_0^*)(S_0^*)^{2(n-m)}\nonumber\\
&+\suml_{n=1}^{\infty} S_0^{2n}(I-S_0 S_0^*)(S_0^*)^{2n+2m-1},
\end{eqnarray}
Notice that, we have $D(m)=D(1)^m,m\in\mathbb{N}$ and Lemma \ref{lem:PropHom} (identity \eqref{eqn:hom_2}) implies that $D(m)$ acts as shift on the sequence $\alpha=\{\alpha_k\}_{k\in\mathbb{Z}}$ (represented by operator $\partial(\alpha)$).

\end{example}

\section*{Acknowledgments}

M.\,N.\ and V.\,F.\ are supported in part by the CNPq (402449/2021-5).


\begin{thebibliography}{99}

\bibitem{BFN2021} W. Bock, V. Futorny, M. Neklyudov, {\em Convex topological algebras via linear vector fields and Cuntz algebras}, J. Pure Appl. Alg.
V. 225, 3, 106535 (2021). 

\bibitem{BratelliJoergensen97}  O. Bratelli, P.T.E. Joergensen, {\em Isometries, shifts, Cuntz algebras and multiresolution wavelet analysis of scale N}, Integ. Equations Oper. Theory,  28, 382–443 (1997).
\bibitem{CHilbert} R. Courant and D. Hilbert, Methods of mathematical physics: First English edition, volume 1, Interscience Publishers, 1953.


\bibitem{CurtainZwart} Curtain, R. F., \& Zwart, H. (2012). An introduction to infinite-dimensional linear systems theory (Vol. 21). Springer Science \& Business Media.

\bibitem{Ding2008} Ding, J., Zhou, A. (2008). A spectrum theorem for perturbed bounded linear operators. Applied mathematics and computation, 201(1-2), 723-728.


\bibitem{Go56} Gonshor, Harry. "Spectral theory for a class of non-normal operators." Canadian Journal of Mathematics 8 (1956): 449-461.

\bibitem{Go58} Gonshor, H. (1958). Spectral theory for a class of nonnormal operators II. Canadian Journal of Mathematics, 10, 97-102.
Chicago	

\bibitem{Hass}  S. Hassani, Mathematical Physics (Springer, New York, 1999).

\bibitem{JacobZwart} Jacob, B., \& Zwart, H. J. (2012). Linear port-Hamiltonian systems on infinite-dimensional spaces (Vol. 223). Springer Science \& Business Media.


\bibitem{eigenfaces} Kshirsagar, V. P., Baviskar, M. R., \& Gaikwad, M. E. (2011, March). Face recognition using Eigenfaces. In 2011 3rd International Conference on Computer Research and Development (Vol. 2, pp. 302-306). IEEE.


\bibitem{M1943} A. Markushevich,
{\em Sur les bases (au sens large) dans les espaces linéaires}, 
C. R. (Doklady) Acad. Sci. URSS (N.S.) 41, 227–-229 (1943).

\bibitem{Tre11} N. Trefethen, (2011). Favorite eigenvalue problems, SIAM News 44(10).

\bibitem{Pazy} Pazy, A. (2012). Semigroups of linear operators and applications to partial differential equations (Vol. 44). Springer Science \& Business Media.

\bibitem{vN1} J. von Neumann, “Zur algebra der funktionaloperationen und theorie der normalen operatoren,” Math. Annalen 102, 370–427 (1930). https://doi.org/10.1007/bf0178235213. 
\bibitem{vN2} J. von Neumann, Mathematical Foundations of Quantum Mechanics (Princeton University Press, Princeton, New Jersey, 1955).
\bibitem{PCA} Vidal, R., Ma, Y., Sastry, S. S., Vidal, R., Ma, Y., \& Sastry, S. S. (2016). Principal component analysis (pp. 25-62). Springer New York.
\end{thebibliography}
\end{document}